\documentclass[12pt]{article}
\usepackage{amsmath}
\usepackage{amsthm}
\usepackage{amsfonts}
\usepackage{amssymb}

\usepackage{hyperref}

\newtheorem{thm}{Theorem}[section]
\newtheorem{lem}[thm]{Lemma}

\theoremstyle{definition}

\newcommand{\Lm}{{\Lambda}}

\let\abs=\envert
\newcommand{\floor}[1]{\left\lfloor#1\right\rfloor}

\theoremstyle{remark}

\begin{document}
\title{On the simultaneous equations $\sigma(2^a)=p^{f_1}q^{g_1}, \sigma(3^b)=p^{f_2}q^{g_2}, \sigma(5^c)=p^{f_3}q^{g_3}$
}
\author{Tomohiro Yamada}
\date{}
\maketitle

\begin{abstract}
We shall solve the simultaneous equations
$\sigma(2^a)=p^{f_1}q^{g_1}, \sigma(3^b)=p^{f_2}q^{g_2}, \sigma(5^c)=p^{f_3}q^{g_3}$
with $p, q$ distinct primes.
\end{abstract}

\section{Introduction}\label{intro}

As usual, let $\sigma(N)$ denote the sum of divisors of $N$ and $\omega(N)$
the number of distinct prime factors of $N$.
In \cite{Ymd}, the author has shown that
there are only finitely many odd superperfect numbers (i.e. the number satisfying $\sigma(\sigma(N))=2N$)
with bounded number of distinct prime factors
by proving that the simultaneous equation $\sigma(p_i^{e_i})=q_1^{f_{1i}}\cdots q_k^{f_{ki}}$
for $k+1$ prime powers $p_i^{e_i} (i=1, 2, \ldots, k+1)$ cannot have solutions with $p_1, \cdots, p_{k+1}$ all small.
In this paper, we use the method developed in \cite{Ymd} to solve the simultaneous equations
$\sigma(2^a)=p^{f_1}q^{g_1}, \sigma(3^b)=p^{f_2}q^{g_2}, \sigma(5^c)=p^{f_3}q^{g_3}$
with $p, q$ distinct primes.

Wakulicz \cite{Wak} has shown that all solutions of the purely exponential diophantine equation
$2^n-5^m=3$ are $(n, m)=(2, 0), (3, 1)$ and $(7, 3)$,
from which Makowski and Schinzel \cite{MS} derived that the equation
$\sigma(2^a)=\sigma(5^c)$ has only the solution $(a, c)=(4, 2)$.
We note that it is easy to show that $\sigma(2^a)=\sigma(3^b)$ has no nontrivial solution and
$\sigma(3^b)=\sigma(5^c)$ also has no nontrivial solution.
Bugeaud and Mignotte \cite{BM} has shown that neither of $\sigma(2^a), \sigma(3^b), \sigma(5^c)$
can be perfect power except $\sigma(3)=2^2$ and $\sigma(3^4)=11^2$.
Moreover, they have shown that the only perfect powers $\frac{x^n-1}{x-1}$ with $x=z^t, z\leq 10$
are $(3^5-1)/2=11^2$ and $(7^4-1)/6=20^2$.

Now we shall state our result.
\begin{thm}\label{th1}
The simultaneous equations $\sigma(2^a)=p^{f_1}q^{g_1}, \sigma(3^b)=p^{f_2}q^{g_2}, \sigma(5^c)=p^{f_3}q^{g_3}$
with $a, b, c>0, f_1, f_2, f_3, g_1, g_2, g_3\geq 0$ and $p, q$ distinct primes
has only the following solutions:
i) $(a, b, c)=(1, 1, 1)$.
ii) $(a, b, c)=(4, 1, 2)$,
iii) $(a, b, c)=(4, 4, 2)$ and
iv) $(a, c)=(4, 2)$ and $\sigma(3^b)$ is prime.
In other words, if $\omega(\sigma(2^a 3^b 5^c))\leq 2$, then $(a, b, c)$ must satisfy
one of the above.
\end{thm}

Our result is related to the Nagell-Ljunggren equation
\begin{equation}
\frac{x^m-1}{x-1}=y^n, x\geq 2, y\geq 2, m\geq 3, n\geq 2,
\end{equation}
which has been conjectured to have only three solutions
$(x, y, m, n)=(3, 11, 5, 2), (7, 20, 4, 2), (18, 7, 3, 3)$.
Some of recent remarkable results concerning to the Nagell-Ljunggren equation
are \cite{BHM}, \cite{BM}, \cite{Mih1}, \cite{Mih2} and \cite{BMh}.
Our result leads us to conjecture that the diophantine equation
\begin{equation}
\frac{x^l-1}{x-1}=y^mz^n
\end{equation}
has only finitely many solutions in integers $x\geq 2, z\geq y\geq 2$ and $l, m, n\geq n_0$
for some constant $n_0$.
The $abc$-conjecture, which Mochizuki \cite{Mch} claims to prove,
would allow us to take $n_0=3$.  Indeed, applying the $abc$-conjecture to the equation
$1+(x-1)y^m z^n=x^l$, we see that for any given $\epsilon>0$,
the inequality
\begin{equation}
\frac{2}{l}+\frac{l-1}{l}\times \frac{2}{n+\min\{n, m\}}\geq 1-\epsilon
\end{equation}
would hold for sufficiently large $x^l$.
Hence, with only finitely many exceptions, we would have
i) $l\geq 5, n=1$, ii) $l=4, n=1$, iii) $(l, m, n)=(4, 1, 2)$.
iv) $l=3, n\leq 2$ or v) $(l, m, n)=(3, 1, 3)$.

\section{Preliminary Lemmas}\label{lemmas}

In this section, we introduce some preliminary lemmas.  One is Matveev's lower bound for linear forms of logarithms \cite{Mat}.
\begin{lem}\label{lmll}
Let $a_1, a_2, \ldots, a_n$ be positive integers such that $\log a_1, \ldots, \log a_n$ are
not all zero and $A_j\geq\max\{0.16, \log a_j\}$ for each $j=1, 2, \ldots, n$.
Moreover, we put
\begin{equation}
\begin{split}
B= & \max \{1, \abs{b_1}A_1/A_n, \abs{b_2}A_2/A_n, \ldots, \abs{b_n} \},\\
\Omega= & A_1A_2\ldots A_n,\\
C_0= & 1+\log 3-\log 2,\\
C_1(n)= & \frac{16}{n!}e^n(2n+3)(n+2)(4(n+1))^{n+1}\left(\frac{1}{2}en\right) \\
& \times (4.4n+5.5\log n+7)
\end{split}
\end{equation}
and
\begin{equation}
\Lm=b_1\log a_1+\ldots+b_n\log a_n.
\end{equation}
Then we have
\begin{equation}
\log\abs{\Lm}>-C_1(n)(C_0+\log B)\max\left\{1, \frac{n}{6}\right\}\Omega.
\end{equation}
\end{lem}

The others concern to some arithmetical properties of values of cyclotomic polynomials.
Lemma \ref{lm21} is a basic and well-known result of this area.
Lemma \ref{lm21} has been proved by Zsigmondy \cite{Zsi} and rediscovered by many authors
such as Dickson \cite{Dic} and Kanold \cite{Kan}.  We need only the special case $b=1$,
where this lemma had already been proved by Bang \cite{Ban}.
See also Theorem 6.4A.1 in \cite{Sha}.
\begin{lem}\label{lm21}
If $a>b\geq 1$ are coprime integers, then $a^n-b^n$ has a prime factor
which does not divide $a^m-b^m$ for any $m<n$, unless $(a, b, n)=(2, 1, 6)$,
$a-b=n=1$, or $n=2$ and $a+b$ is a power of $2$.
\end{lem}

Let $o_p(a)$ denote the residual order of $a\pmod{p}$.
Lemma \ref{lm21} immediately gives the following result.
\begin{lem}\label{lm22}
If $(a^e-1)/(a-1)=p^{f_1}q^{f_2}$ for some integers $a, e, f_1, f_2$
and primes $p<q$, then we have $(a, e, p, q, f_1, f_2)=(2, 6, 3, 7, 2, 1)$,
$e=r$ or $e=r^2$ for some prime $r$.  Moreover, in the case $e=r$, then we have $p=r, o_q(a)=r$
or $o_p(a)=o_q(a)=r$.
In the case $e=r^2$, we have $(p, q, f_1, f_2)=((a^r-1)/(a-1), (a^{r^2}-1)/(a^r-1), 1, 1)$
or $(a, e, p, f_1)=(2^m-1, 4, 2, m+1)$ for some integer $m$.
\end{lem}

The following lemma is proved in \cite{BM}, as mentioned in the introduction.
\begin{lem}\label{lm23}
Let $a, e, x, f$ be positive integers with $a, x, f>1$ and $e>2$.
The equation $(a^e-1)/(a-1)=x^f$ has no solution but $(a, e, x, f)=(3, 5, 11, 2), (7, 4, 20, 2)$
in integers $2\leq a\leq 10, e>2, x>1, f>1$.
\end{lem}

Using results mentioned in the introduction,
we can immediately solve some special case of our main theorem.
\begin{lem}\label{lm24}
Choose $a<b$ from the first three primes $2, 3, 5$.
If $(a^e-1)/(a-1)=p^k$ and $(b^f-1)/(b-1)=p^l$ for some integers $e, f, k, l$
and some prime $p$, then $(a^e, b^f)=(2^5, 5^3)$ and $p=31, k=l=1$.
\end{lem}
\begin{proof}
In the case $k=l=1$ and $(a^e-1)/(a-1)=(b^f-1)/(b-1)$,
as observed in the introduction, we have $(a^e, b^f)=(2^5, 5^3)$.

Lemma \ref{lm23} yields that the perfect power case
must arise from $(3^5-1)/2=11^2$ or $(3^2-1)/2=2^2$.
In this case, we must have $2^e-1=2$ or $11$ or $(5^f-1)/4=2$ or $11$,
which is clearly impossible.
\end{proof}

\section{Main Theory}

For convenience, we put $a_1=2, a_2=3, a_3=5$ and $e_1=a+1, e_2=b+1, e_3=c+1$.

\begin{lem}\label{lm31}
For each $i=1, 2, 3$, we have
\begin{equation}\label{eq11}
e_i\log a_i<E_i=C_i\log p\log q (\log\log p+C_{i+3}),
\end{equation}
where
$C_1=1.422\times 10^{10}, C_2=1.226\times 10^{12}, C_3=1.795\times 10^{12},
C_4=23.3, C_5=27.8, C_6=28.1$.
\end{lem}
\begin{proof}
Let $\Lm_i=f_i\log a_i+g_i\log q+\log (a_i-1)-e_i\log 2=\log(1-a_i^{-e_i})$
for $i=1, 2, 3$.  It immediately follows from Matveev's theorem that
\begin{equation}
-\log\abs{\Lm_1}<C(3)\left(C_0+\log\frac{e_1\log 2}{\log q}\right)\log 2\log p\log q
\end{equation}
and
\begin{equation}
-\log\abs{\Lm_j}<C(4)\left(C_0+\log\frac{e_j\log a_j}{\log q}\right)\log 2\log a_j\log p\log q
\end{equation}
for $j=2, 3$.

Now we shall show (\ref{eq11}) in the case $i=1$.
We may assume that $e_1>10^{10} \log q/\log 2$.
Since $0<\abs{\Lm_1}=-\log(1-2^{-e_1})<\frac{1}{2^{e_1}-1}$, we have
$-\log\abs{\Lm_1}>\log(2^{e_1}-1)>(1-10^{-10})e_1\log 2$.
Combining upper and lower bounds for $\Lm_1$, we obtain
\begin{equation}
\begin{split}
\frac{e_1\log 2}{\log q}< & (1+10^{-10})\left(C_0+\log\frac{e_1\log 2}{\log q}\right)C(3)\log 2\log p \\
< & 1.244\times 10^{10}\log p\log\frac{e_1\log 2}{\log q}.
\end{split}	
\end{equation}
Hence, observing that $1.244\times 10^{10}\log p\geq 1.244\times 10^{10}\log 2$, we obtain
\begin{equation}
\begin{split}
\frac{e_1\log 2}{\log q}<& 1.143\times (1.244\times 10^{10}\log p)\log(1.244\times 10^{10}\log p) \\
<& 1.422\times 10^{10}(\log\log p+23.3),
\end{split}
\end{equation}
giving (\ref{eq11}) in the case $i=1$.

Next we shall prove (\ref{eq11}) in the case $i=2$.
We may assume that $e_2>10^{10} \log q/\log 3$ as with the previous case.
Since $0<\abs{\Lm_2}=-\log(1-3^{-e_2})<\frac{1}{3^{e_2}-1}$, we have
$-\log\abs{\Lm_2}>\log(3^{e_2}-1)>(1-10^{-10})e_2\log 3$.
From $0<\abs{\Lm_2}=-\log(1-3^{-e_2})<\frac{1}{3^{e_2}-1}$, we see that
$-\log\abs{\Lm_2}>\log(3^{e_2}-1)\geq (1-10^{-10})e_2\log 3$ and therefore
\begin{equation}
\begin{split}
\frac{e_2\log 3}{\log q}< & (1+10^{-10})\left(C_0+\log\frac{e_2\log 3}{\log q}\right)C(4)\log 2\log 3\log p \\
< & 1.089\times 10^{12}\log p\log\frac{e_2\log 3}{\log q}.
\end{split}
\end{equation}
This gives (\ref{eq11}) in the case $i=2$.

Similarly, (\ref{eq11}) in the case $i=3$ follows from
\begin{equation}
\begin{split}
\frac{e_3\log 5}{\log q}< & (1+10^{-10})\left(C_0+\log\frac{e_3\log 5}{\log q}\right)C(4)\log 2\log 5\log p \\
< & 1.595\times 10^{12}\log p\log\frac{e_3\log 5}{\log q}.
\end{split}
\end{equation}

A similar argument yields (\ref{eq11}) in the case $i=3$.  This completes the proof of the lemma.
\end{proof}

Next, we shall show that we cannot have all of $a_i^{e_i}$'s large.

\begin{lem}\label{lm32}
Let $x$ be the smallest among $a_i^{e_i}$'s.
Let $h_1=f_2g_3-f_3g_2, h_2=f_3g_1-f_1g_3$ and $h_3=f_1g_2-f_2g_1$ and
$H=\max\abs{h_i}$.  Then
\begin{equation}\label{eq33}
\log x\leq \log\frac{7H}{4}+C(3)(C_0+\log((e_1+3)H))\log 2\log 3\log 5.
\end{equation}
\end{lem}
\begin{proof}
We begin by observing that
\begin{equation}
(2^{e_1}-1)^{h_1}\left(\frac{3^{e_2}-1}{2}\right)^{h_2}\left(\frac{5^{e_3}-1}{4}\right)^{h_3}=1.
\end{equation}
Now we put
\begin{equation}
\begin{split}
\Lm&=(e_1h_1-h_2-2h_3)\log 2+e_2h_2\log 3+e_3h_3\log 5\\
&=h_1\log\frac{2^{e_1}}{2^{e_1}-1}+h_2\log\frac{3^{e_2}}{3^{e_2}-1}+h_3\log\frac{5^{e_3}}{5^{e_3}-1}.
\end{split}
\end{equation}
Then we have
\begin{equation}
0<\abs{\Lm}\leq H\left(\frac{1}{2^{e_1}-1}+\frac{1}{3^{e_2}-1}+\frac{1}{5^{e_3}-1}\right)\leq \frac{7H}{4x}
\end{equation}
and therefore
\begin{equation}\label{eq31}
\log\abs{\Lm}\leq -\log x+\log\frac{7H}{4}.
\end{equation}

It follows from the assumption $e_i>0$ that $\Lm\neq 0$.  Hence Matveev's lower bound gives
\begin{equation}\label{eq32}
\log\abs{\Lm}\geq -C(3)(C_0+\log((e_1+3)H))\log 2\log 3\log 5.
\end{equation}

Combining (\ref{eq31}) and (\ref{eq32}), we obtain (\ref{eq33}).
\end{proof}

The third step is to obtain upper bounds for each $e_i$.

\begin{lem}\label{lm33}
Unless $x=p=31$, we have $e_1<4.44\times 10^{52}, e_2<2.54\times 10^{54}$ and $e_3<2.55\times 10^{54}$
and $H<2.89\times 10^{68}$.
\end{lem}
\begin{proof}
We may assume without the loss of generality that $p<q$.
We begin by considering $q\mid x$.  In this case, we have
\begin{equation}
\log q<\log x<\log\frac{7H}{4}+C(3)(C_0+\log((e_1+3)H))\log 2\log 3\log 5.
\end{equation}
We note that $H\leq C_2C_3\log p\log q(\log\log p+C_5)(\log\log p+C_6)$
since it follows from Lemma \ref{lm31} that $f_i<e_i\log a_i/\log p_i<C_i\log q(\log\log p+C_{i+3})$
and $g_i<e_i\log a_i/\log q_i<C_i\log p(\log\log p+C_{i+3})$.
Hence we obtain $\log p<\log q<4.35\times 10^{12}$.

Now we consider the case $p<q$ and $q\nmid x$.
Put $i$ to be the index such that $x=(a_i^{e_i}-1)/(a_i-1)$, $j, k$ be the others
and
\begin{equation}
\begin{split}
\Lm^\prime & =e_j h_j\log a_j+e_k h_k\log a_k-h_j\log(a_j-1)-h_k\log(a_k-1)+h_i\log x \\
& =h_j\log\frac{a_j^{e_j}}{a_j^{e_j}-1}+h_k\log\frac{a_k^{e_k}}{a_k^{e_k}-1}.
\end{split}
\end{equation}

It follows from Lemma \ref{lm24} that if $(a_j^{e_j}-1)/(a_j-1)=p^{f_j}$
or $(a_k^{e_k}-1)/(a_k-1)=p^{f_k}$, then $a_i^{e_i}=2^5$ or $5^3$ and $x=p=31$.
Hence we see that both numbers $(a_j^{e_j}-1)/(a_j-1), (a_k^{e_k}-1)/(a_k-1)$
must be divisible by $q$ unless $x=p=31$.

Thus we obtain
\begin{equation}\label{eq34}
0<\Lm^\prime<H\left(\frac{1}{a_j^{e_j}-1}+\frac{1}{a_k^{e_k}-1}\right)\leq \frac{3H}{2q}.
\end{equation}
As in the previous case, Matveev's theorem now gives
\begin{equation}\label{eq35}
\log\abs{\Lm^\prime}\geq -C(4)\left(C_0+\log\frac{E_3H}{\log x}\right)\log 2\log 3\log 5\log x.
\end{equation}

Combining (\ref{eq34}) and (\ref{eq35}), we obtain
\begin{equation}\label{eq36}
\log q\leq \log\frac{3H}{2}+C(4)\left(C_0+\log\frac{E_3H}{\log x}\right)\log 2\log 3\log 5\log x.
\end{equation}
Since $E_3=C_3\log p\log q (\log\log p+C_6)\leq C_3\log x\log q(\log\log x+C_6)$
and $H<C_2C_3(\log q)^2(\log\log q+C_5)(\log\log q+C_6)$,
combining (\ref{eq33}) and (\ref{eq36}), we obtain $\log q<3.45\times 10^{27}$.
Moreover, we have
\begin{equation}
\begin{split}
\log p= & \log x \\
< & \log\frac{7H}{4}+C(3)(C_0+\log((e_1+3)H))\log 2\log 3\log 5 \\
< & 7.22\times 10^{12}.
\end{split}
\end{equation}

So that, we conclude that in both cases, we have $\log p<7.22\times 10^{12}$ and $\log q<3.45\times 10^{27}$.
Now Lemma \ref{lm31} immediately gives that $e_1<4.44\times 10^{52}, e_2<2.54\times 10^{54}$ and $e_3<2.55\times 10^{54}$.
Finally, the upper bound $H<2.89\times 10^{68}$ follows from $H<C_2C_3(\log p)(\log q)(\log\log p+C_6)(\log\log q+C_5)$.
\end{proof}

Now, using the lattice reduction algorithm, we shall obtain feasible upper bounds.
\begin{lem}\label{lm34}
We have $\log x<354.8$.
Moreover, if $p<q$ and $q$ divides $x$, then $\log x<249.5$.
\end{lem}
\begin{proof}
We begin by noting that we can assume $x\neq 31$ without the loss of generality.

In order to reduce our upper bounds, we use the LLL lattice reduction algorithm
introduced in \cite{LLL}.
Let $M$ be the matrix defined by $m_{12}=m_{13}=m_{21}=m_{23}=0$
and $m_{11}=m_{22}=\gamma$ and $m_{3i}=\floor{C\gamma\log a_i}$ for $i=1, 2, 3$,
where $C$ and $\gamma$ are constants chosen later.
Let $L$ denote the LLL-reduced matrix of $M$ and $l(L)$ the shortest length
of vectors in the lattice generated by the column vectors of $L$.

From the previous lemma, we know that $\Lm$ has coefficients with absolute values
at most $H\max\{e_1+3, e_2, e_3\}<7.37\times 10^{122}$.
It is implicit in the proof of Lemma 3.7 of de Weger's book \cite{Weg}
that if $l(\Gamma)>X_1\sqrt{16+4\gamma}$ and $X_1\geq 7.37\times 10^{122}$, then $\abs{\Lm}>X_1/(C\gamma)$.

Taking $C=10^{370}, \gamma=2$, we can confirm that $l(\Gamma)>X_1\sqrt{16+4\gamma}$
and therefore we obtain that $\abs{\Lm}>3.685\times 10^{-248}$.
Hence we have
\begin{equation}
\log x<\log\frac{7H}{4}-\log\abs{\Lm}<727.94.
\end{equation}
We choose the index $i$ such that $x=(a_i^{e_i}-1)/(a_i-1)$ and let $j, k$ be the others.
From the above estimate for $x$, we derive that
\begin{equation}\label{eq37}
e_i\leq \floor{\frac{\log 2x}{\log a_i}}\leq 1051.
\end{equation}

We consider the case $p<q$ and $q$ does not divide $x$.
From (\ref{eq36}) we obtain $\log q<3.337\times 10^{17}$.
Lemma \ref{lm31} gives that
\begin{equation}
\begin{split}
\abs{h_i}< & C_2 C_3\log x\log q(\log\log x+C_6)(\log\log q+C_5)\\
< & 1.264\times 10^{48},
\end{split}
\end{equation}
\begin{equation}
\abs{h_j}=\abs{f_ig_k}<C_3\log x(\log\log q+C_6)<8.944\times 10^{16},
\end{equation}
\begin{equation}
\abs{e_j}<C_3\log x\log q(\log\log q+C_6)/\log 2<4.306\times 10^{34}
\end{equation}
and similar upper bounds hold for $\abs{h_k}$ and $\abs{e_k}$, respectively.
Hence $\Lm$ has coefficients with absolute values at most $3.852\times 10^{51}$.
Using the LLL-reduction again with $C=10^{157}$ and $\gamma=2$,
we obtain $\abs{\Lm}>1.926\times 10^{-106}$
and therefore $\log x<\log(7H/4)-\log \abs{\Lm}<354.8$.

Next, we consider the case $p<q$ and $q$ divides $x$.
In this case, we have $\log p<\log q\leq \log x<727.94$.
We choose the index $i$ such that $x=(a_i^{e_i}-1)/(a_i-1)$ and let $j, k$ be the others.
Lemma \ref{lm31} gives that
\begin{equation}\label{eq38}
\abs{h_i}<C_2 C_3\log^2 x(\log\log x+C_5)(\log\log x+C_6)<1.392\times 10^{33},
\end{equation}
\begin{equation}
\abs{h_j}\leq \max\abs{f_ig_k, f_kg_i}<C_3\log x(\log\log x+C_6)<4.533\times 10^{16},
\end{equation}
\begin{equation}\label{eq39}
\abs{e_j}<C_3\log^2 x(\log\log x+C_6)/\log 2<4.761\times 10^{19}
\end{equation}
and similar upper bounds hold for $\abs{h_k}$ and $\abs{e_k}$, respectively.
Combining these upper bounds with (\ref{eq37}),
we see that $\Lm$ has coefficients with absolute values at most $2.159\times 10^{36}$.
We use the LLL-reduction again with $C=10^{111}$ and $\gamma=2$,
we obtain $\abs{\Lm}>1.079\times 10^{-75}$
and therefore $\log x<\log(7H/4)-\log \abs{\Lm}<249.5$.  This proves the lemma.
\end{proof}

\section{The final step}

The final step is checking all possibilities of $x$.
We note that from The Cunningham Project (see \cite{Wag} or \cite{Cro}),
we know all prime factors of $x$'s below our upper bounds.

For $x=(a_i^{e_i}-1)/(a_i-1)$, we should check the residual orders of
the other prime $a_j$ modulo $x$.  A summary is given in Tables \ref{tbl1}-\ref{tbl6},
where $\textrm{P}n$ denotes a prime with $n$ digits and $(n)$ indicates that
the residual order is a multiple of $n$.  For example, putting $x=2^{347}-1=pq$ with $p<q$,
$o_q(3)$ is divisible by $6$ since $q-1$ is divisible by $2^3\times 3^2$
and $3^{(q-1)/8}, 3^{(q-1)/3}\not\equiv 1\pmod{q}$, although $3^{(q-1)/4}\equiv 1\pmod{q}$,
which yields that $(3^{e_2}-1)/2=p^{f_2}q^{g_2}$ with $g_2>0$ is impossible.

If $p=x=2^{e_1}-1$ is prime, then $e_1\leq 511$
and therefore $e_1$ must belong to the set
\begin{equation*}
\{2, 3, 5, 7, 13, 17, 19, 31, 61, 89, 107, 127\}.
\end{equation*}
Among them, there exists no $e_1$ such that the $o_x(3)=1$, a prime or the square of a prime,
as we can see from Table \ref{tbl1}.
Hence, by Lemma \ref{lm22}, we must have $(3^{e_2}-1)/2=q^{g_2}$.  By Lemma \ref{lm24},
$(5^{e_3}-1)/4$ must be divisible by $p=x$.
Hence, by Lemma \ref{lm22}, $o_x(5)=1$ or $o_x(5)$ must be a prime or a prime square
and therefore, from Table \ref{tbl1}, $e_1=e_3=2$ or $e_1=5, e_3=3$.
If $e_1=e_3=2$, then $(5^{e_3}-1)/4=6=2\times 3$ and therefore $(3^{e_2}-1)/2$ must be a power of $2$,
yielding that $e_2=2$.  If $e_1=5$ and $e_3=3$, then $p=31$ and $(3^{e_2}-1)/2=q^{g_2}$,
yielding that $e_2=5$ or $(3^{e_2}-1)/2=q$.

If $x=2^{e_1}-1$ is not a prime power, then $e_1\leq 359$ and therefore
$e_1$ must belong to the set
\begin{equation*}
\begin{split}
\{& 4, 6, 9, 11, 23, 37, 41, 49, 59, 67, 83, 97, 101, 103, 109, 131, 137, 139, 149, \\
& 167, 197, 199, 227, 241, 269, 271, 281, 293, 347\}.
\end{split}
\end{equation*}
Hence we can write $x=2^{e_1}-1=pq$ for distinct primes $p<q$.
By Lemma \ref{lm24}, $(3^{e_2}-1)/2=p^{g_2}$ and $(5^{e_3}-1)/4=p^{g_3}$ cannot simultaneously
hold.  In other words, at least one of these two integers must be divisible by $q$.
But, for no $e_1$ in the above set, $o_q(5)$ is $1$, a prime or prime-square,
as can be seen from Table \ref{tbl2}.
Hence $(3^{e_2}-1)/2$ must be divisible by $q$.  The only $e_1$ for which
$o_q(3)$ is $1$, a prime or prime-square is $e_1=4$.
Then we must have $x=2^4-1=3\times 5$ and $(p, q)=(3, 5)$.
But this implies that $e_2$ is divisible by $4$ and $(3^{e_2}-1)/2$ must
be divisible by $2$.  Hence $(3^{e_2}-1)/2$ cannot be of the form $p^{f_2}q^{g_2}$.
Hence it cannot occur that $x=2^{e_1}-1$ is not a prime power.

If $x=(3^{e_2}-1)/2=p^{f_2}$ is prime or prime power, then
\begin{equation*}
e_2\in\{2, 3, 5, 7, 13, 71, 103\}.
\end{equation*}
For none of them, $o_p(2)=1, 6$ or a prime power.
Hence, as above cases, $(5^{e_3}-1)/4$ must be divisible by $p$.
Since $o_p(5)$ must be $1$ or a prime power, we must have $e_2\in \{2, 3, 5\}$.
If $e_2=2$, then $p=2$ and $e_3=2$, which yields that $q=3$ and $e_1=2$.
If $e_2=3$, then $p=13$ and $e_3=4$, which is impossible since $(5^{e_3}-1)/4=156=2^2\times 3\times 13$
has three distinct prime factors.
If $e_2=5$, then $p=11$ and $e_3=5$.  Hence $(5^{e_3}-1)/4=781=11\times 71$.
This implies that $2^{e_1}-1=11^{f_1}71^{g_1}$, which is impossible
since $2^{10}-1=3\time 11\times 31$ and $2^{35}-1=31\times 71\times 127\times 122921$.

If $x=(3^{e_2}-1)/2$ is not a prime power, then
\begin{equation*}
e_2\in\{9, 11, 17, 19, 23, 37, 43, 59, 61, 223\}.
\end{equation*}
Hence we can write $x=(3^{e_1}-1)/2=pq$ for distinct primes $p<q$ with $p, q\neq 31$.
However, $o_q(2)$ or $o_q(5)$ can never be $1, 6$ or a prime power
among the above $e_2$'s.  Hence both $2^{e_1}-1$ and $(5^{e_3}-1)/4$
must be a power of $p$.  By Lemma \ref{lm24}, we must have $p=31$,
which is impossible as mentioned above.

If $x=(5^{e_3}-1)/4$ is a prime power, then
\begin{equation*}
e_3\in \{3, 7, 11, 13, 47, 127, 149, 181\}.
\end{equation*}
Among them, no $e_3$ gives a prime power (or one) residual order $3\pmod{x}$
and only $e_3=3$ makes the residual order $2\pmod{x}$ acceptable in view of Lemma \ref{lm22}.
Hence $p=31, e_3=3, e_1=5$ and $(3^{e_2}-1)/2=q^{f_2}$, which implies that
$e_2=5$ or $(3^{e_2}-1)/2=q$.

If $x=(5^{e_3}-1)/4$ is not a prime power, then
\begin{equation*}
e_3\in \{2, 5, 17, 23, 31, 41, 43, 59, 71\}.
\end{equation*}
Hence we can write $x=(5^{e_3}-1)/4=pq$ for distinct primes $p<q$.
None of such $e_3>2$ gives an acceptable residual order $2\pmod{q}$ or $3\pmod{q}$
in view of Lemma \ref{lm22}.
Hence we see that neither $2^{e_1}-1$ nor $(3^{e_2}-1)/2$ can be divisible by $q$
and both must be a power of $p$, contrary to Lemma \ref{lm24}.
Hence we must have $e_3=2, (p, q)=(2, 3)$.  This yields that $e_1=e_2=2$.

This completes the proof of Theorem \ref{th1}.

\begin{table}
\caption{The residual orders of $3, 5$ modulo $p$ for $p=2^{e_1}-1$}\label{tbl1}
\begin{center}
\begin{small}
\begin{tabular}{| c | c | c |}
 \hline
$e_1$ & $o_p(3)$ & $o_p(5)$ \\
 \hline
$2$ & N/A & $2$ \\
$3$ & $6$ & $6$ \\
$5$ & $30$ & $3$ \\
$7$ & $126$ & $42$ \\
$13$ & $910$ & $1365$ \\
$17$ & $131070$ & $65535$ \\
$19$ & $524286$ & $74898$ \\
$31$ & $715827882$ & $195225786$ \\
$61$ & $(10)$ & $(15)$ \\
$89$ & $(6)$ & $(84)$ \\
$107$ & $(6)$ & $(6)$ \\
$127$ & $(6)$ & $(6)$ \\
 \hline
\end{tabular}
\end{small}
\end{center}
\end{table}

\begin{table}
\caption{The residual orders of $3, 5$ modulo $p, q$ for $pq=2^{e_1}-1, p<q$}\label{tbl2}
\begin{center}
\begin{small}
\begin{tabular}{| c | c | c | c | c | c |}
 \hline
$e_1$ & $2^{e_1}-1=pq$ & $o_p(3)$ & $o_q(3)$ & $o_p(5)$ & $o_q(5)$ \\
 \hline
$9$ & $7\times 73$ & $6$ & $12$ & $6$ & $72$ \\
$11$ & $23\times 89$ & $11$ & $88$ & $22$ & $44$ \\
$23$ & $47\times 178481$ & $23$ & $178480$ & $46$ & $44620$ \\
$37$ & $223\times 616318177$ & $222$ & $308159088$ & $222$ & $616318176$ \\
$41$ & $13367\times 164511353$ & $6683$ & $164511352$ & $13366$ & $164511352$ \\
$49$ & $127\times \textrm{P13}$ & $126$ & $(8)$ & $42$ & $(8)$ \\
$59$ & $179951\times \textrm{P13}$ & $89975$ & $(8)$ & $89975$ & $(8)$ \\
$67$ & $193707721\times \textrm{P12}$ & $96853860$ & $(6)$ & $8071155$ & $(6)$ \\
$83$ & $167\times \textrm{P23}$ & $83$ & $(10)$ & $166$ & $(166)$ \\
$97$ & $11447\times \textrm{P26}$ & $5723$ & $(194)$ & $11446$ & $(194)$ \\
$101$ & $\textrm{P13}\times \textrm{P14}$ & $(303)$ & $(303)$ & $(303)$ & $(303)$ \\
$103$ & $2550183799\times \textrm{P22}$ & $(166)$ & $(206)$ & $(249)$ & $(309)$ \\
$109$ & $745988807\times \textrm{P24}$ & $(11663)$ & $(118)$ & $(214)$ & $(118)$ \\
$131$ & $263\times \textrm{P38}$ & $131$ & $(74)$ & $262$ & $(74)$ \\
$137$ & $\textrm{P20}\times \textrm{P22}$ & $(274)$ & $(66290053)$ & $(1202723)$ & $(66290053)$ \\
$139$ & $\textrm{P13}\times \textrm{P30}$ & $(6)$ & $(6)$ & $(6)$ & $(15)$ \\
$149$ & $\textrm{P20}\times \textrm{P25}$ & $(745)$ & $(16)$ & $(745)$ & $(8)$ \\
$167$ & $2349023\times \textrm{P44}$ & $(26)$ & $(22)$ & $(26)$ & $(22)$ \\
$197$ & $7487\times \textrm{P56}$ & $(3743)$ & $(394)$ & $(38)$ & $(394)$ \\
$199$ & $\textrm{P12}\times \textrm{P49}$ & $(14)$ & $(1393)$ & $(8)$ & $(1393)$ \\
$227$ & $\textrm{P18}\times \textrm{P52}$ & $(8)$ & $(35)$ & $(8)$ & $(497)$ \\
$241$ & $22000409\times \textrm{P66}$ & $(8)$ & $(5114261)$ & $(482)$ & $(5114261)$ \\
$269$ & $13822297\times \textrm{P74}$ & $(6)$ & $(6)$ & $(6)$ & $(22)$ \\
$271$ & $15242475217\times \textrm{P72}$ & $(8)$ & $(542)$ & $(8)$ & $(15)$ \\
$281$ & $80929\times \textrm{P80}$ & $(8)$ & $(278)$ & $(6)$ & $(417)$ \\
$293$ & $\textrm{P26}\times \textrm{P63}$ & $(6)$ & $(6)$ & $(8)$ & $(6)$ \\
$347$ & $\textrm{P23}\times \textrm{P82}$ & $(6)$ & $(6)$ & $(21)$ & $(8)$ \\
 \hline
\end{tabular}
\end{small}
\end{center}
\end{table}

\begin{table}
\caption{The residual orders of $2, 5$ modulo $p$ for $p^{f_2}=(3^{e_2}-1)/2$}\label{tbl3}
\begin{center}
\begin{small}
\begin{tabular}{| c | c | c |}
 \hline
$e_2$ & $o_p(2)$ & $o_p(5)$ \\
 \hline
$2$ & N/A & $1$ \\
$3$ & $12$ & $4$ \\
$5$ & $10$ & $5$ \\
$7$ & $1092$ & $364$ \\
$13$ & $398580$ & $30660$ \\
$71$ & $(8)$ & $(8)$ \\
$103$ & $(12)$ & $(14)$ \\
 \hline
\end{tabular}
\end{small}
\end{center}
\end{table}

\begin{table}
\caption{The residual orders of $2, 5$ modulo $p, q$ for $pq=(3^{e_2}-1)/2, p<q$}\label{tbl4}
\begin{center}
\begin{small}
\begin{tabular}{| c | c | c | c | c | c |}
 \hline
$e_2$ & $(3^{e_2}-1)/2=pq$ & $o_p(2)$ & $o_q(2)$ & $o_p(5)$ & $o_q(5)$ \\
 \hline
$9$ & $13\times 757$ & $12$ & $756$ & $4$ & $756$ \\
$11$ & $23\times 3851$ & $11$ & $3850$ & $22$ & $1925$ \\
$17$ & $1871\times 34511$ & $935$ & $595$ & $935$ & $3451$ \\
$19$ & $1597\times 363889$ & $532$ & $181944$ & $532$ & $22743$ \\
$23$ & $47\times 1001523179$ & $23$ & $(46)$ & $46$ & $(1073)$ \\
$37$ & $13097927\times \textrm{P12}$ & $(9731)$ & $8594564351$ & $(74)$ & $(74)$ \\
$43$ & $431\times \textrm{P18}$ & $43$ & $215$ & $(22)$ & $(22)$ \\
$59$ & $14425532687\times \textrm{P18}$ & $(3953)$ & $(118)$ & $(106)$ & $(10679)$ \\
$61$ & $603901\times \textrm{P24}$ & $201300$ & $(12)$ & $150975$ & $(145)$ \\
$223$ & $\textrm{P26}\times \textrm{P81}$ & $(446)$ & $(12)$ & $(6)$ & $(446)$ \\
 \hline
\end{tabular}
\end{small}
\end{center}
\end{table}

\begin{table}
\caption{The residual orders of $2, 3$ modulo $p$ for $p=(5^{e_3}-1)/4$}\label{tbl5}
\begin{center}
\begin{small}
\begin{tabular}{| c | c | c |}
 \hline
$e_3$ & $o_p(2)$ & $o_p(3)$ \\
 \hline
$3$ & $5$ & $30$ \\
$7$ & $6510$ & $6510$ \\
$11$ & $1220703$ & $369910$ \\
$13$ & $61035156$ & $1211015$ \\
$47$ & $(94)$ & $(6)$ \\
$127$ & $(18)$ & $(18)$ \\
$149$ & $(10)$ & $(6)$ \\
$181$ & $(12)$ & $(15)$ \\
 \hline
\end{tabular}
\end{small}
\end{center}
\end{table}

\begin{table}
\caption{The residual orders of $2, 3$ modulo $p, q$ for $pq=(5^{e_3}-1)/4, p<q$}\label{tbl6}
\begin{center}
\begin{small}
\begin{tabular}{| c | c | c | c | c | c |}
 \hline
$e_3$ & $(5^{e_3}-1)/4=pq$ & $o_p(2)$ & $o_q(2)$ & $o_p(3)$ & $o_q(3)$ \\
 \hline
$2$ & $2\times 3$ & N/A & $2$ & $1$ & N/A \\
$5$ & $11\times 71$ & $10$ & $35$ & $5$ & $35$ \\
$17$ & $409\times 466344409$ & $204$ & $3429003$ & $204$ & $116586102$ \\
$23$ & $8971\times \textrm{P12}$ & $8970$ & $(8)$ & $8970$ & $2306995565$ \\
$31$ & $1861\times \textrm{P18}$ & $1860$ & $(15)$ & $310$ & $(6)$ \\
$41$ & $2238236249\times \textrm{P19}$ & $279779531$ & $(8)$ & $(8)$ & $(8)$ \\
$43$ & $1644512641\times \textrm{P21}$ & $(8)$ & $(15)$ & $(8)$ & $(10)$ \\
$59$ & $\textrm{P17}\times \textrm{P25}$ & $(12)$ & $(9)$ & $(6)$ & $(118)$ \\
$71$ & $569\times \textrm{P47}$ & $284$ & $(142)$ & $568$ & $(452610863706241)$ \\
 \hline
\end{tabular}
\end{small}
\end{center}
\end{table}

{}

{\small Center for Japanese language and culture, Osaka University,\\ 562-8558, 8-1-1, Aomatanihigashi, Minoo, Osaka, Japan}\\
{\small e-mail: \protect\normalfont\ttfamily{tyamada1093@gmail.com} \url{http://tyamada1093.web.fc2.com/math/}
\end{document}